\documentclass[11pt,a4paper,reqno]{amsart}

\usepackage{amsfonts}
\usepackage{amsmath}
\usepackage{amssymb}
\usepackage{amsthm}
\usepackage{amsxtra}
\usepackage{bbm}
\usepackage{braket}
\usepackage{comment}
\usepackage{extarrows}
\usepackage{graphicx}
\usepackage{latexsym}
\usepackage{mathrsfs}
\usepackage{yhmath}
\usepackage{nicematrix}
\usepackage{mathtools}

\usepackage{float}
\usepackage{booktabs}
\usepackage{verbatim}
\usepackage{xcolor}
\usepackage{caption,subcaption}

\usepackage[pdfborder={0 0 0}]{hyperref} 
\hypersetup{breaklinks, raiselinks}

\usepackage{bm}

\usepackage{tikz}
\usetikzlibrary{arrows, patterns, bending}
\tikzstyle{ghostfill} = [fill=white]
\tikzstyle{ghostdraw} = [draw=black!50]
\usetikzlibrary{arrows,shapes,positioning}
\usetikzlibrary{decorations.markings}
\tikzstyle arrowstyle=[scale=1]
\tikzstyle directed=[postaction={decorate,decoration={markings, mark=at position .65 with {\arrow[arrowstyle]{stealth}}}}]
\tikzstyle reverse directed=[postaction={decorate,decoration={markings, mark=at position .65 with {\arrowreversed[arrowstyle]{stealth};}}}]

\usepackage[initials,lite]{amsrefs} 
\AtBeginDocument{%
    \def\MR#1{}
}

\usepackage{cleveref}

\usepackage{a4wide}
\usepackage{fullpage}

\theoremstyle{plain}
\newtheorem{Theorem}{Theorem}[section]
\newtheorem{Lemma}[Theorem]{Lemma}
\newtheorem{Corollary}[Theorem]{Corollary}

\theoremstyle{definition}

\newtheorem{Assumptions and Discussion}[Theorem]{Assumptions and Discussion}

\newtheorem{Definition}[Theorem]{Definition}

\newtheorem{Remark}[Theorem]{Remark}

\newtheorem{Observation}[Theorem]{Observation}

\theoremstyle{remark}

\newtheorem*{acknowledgment*}{Acknowledgment}

\usepackage[shortlabels]{enumitem}
\SetEnumerateShortLabel{a}{\textup{(\alph*)}}
\SetEnumerateShortLabel{A}{\textup{(\Alph*)}}
\SetEnumerateShortLabel{1}{\textup{(\arabic*)}}
\SetEnumerateShortLabel{i}{\textup{(\roman*)}}
\SetEnumerateShortLabel{I}{\textup{(\Roman*)}}
\setlist{leftmargin=*}

\def\bar#1{\overline{#1}}

\def\Char{\operatorname{char}}

\def\deg{\operatorname{deg}}

\def\Ht{\operatorname{ht}} 

\def\iv{\operatorname{iv}} 

\def\oc{\operatorname{oc}}

\def\reg{\operatorname{reg}}

\newcommand\calI{\mathcal{I}}

\newcommand{\pd}{\operatorname{pd}}

\begin{document}

\title{Regularity of powers of (parity) binomial edge ideals}
\date{\today}

\author{Yi-Huang Shen}
\address{CAS Wu Wen-Tsun Key Laboratory of Mathematics, School of Mathematical Sciences, University of Science and Technology of China, Hefei, Anhui, 230026, P.R.~China}
\email{yhshen@ustc.edu.cn}

\author{Guangjun Zhu$^{\ast}$}
\address{School of Mathematical Sciences, Soochow University, Suzhou, Jiangsu, 215006, P.R.~China}
\email{zhuguangjun@suda.edu.cn}

\thanks{$^{\ast}$ Corresponding author}


\begin{abstract}
    In this paper, we provide exact formulas for the Castelnuovo-Mumford regularity of powers of an almost complete intersection ideal $I$ which is generated by a homogeneous $d$-sequence. As applications, when $I$ is an almost complete intersection, taking the form of the (parity) binomial edge ideal of a connected graph, we can describe explicitly formulas for  $\reg(I^t)$ for $t\ge 2$. The only exception is when $I$ is the parity binomial edge ideal of a graph which is obtained by adding an edge between two disjoint odd cycles.
\end{abstract}

\maketitle

\section{Introduction}
In this article, we are interested in the Castelnuovo-Mumford regularity (abbreviated as  regularity) of powers of homogeneous ideals in the polynomial rings. Let $R=\kappa[z_{1},\dots,z_{n}]$ be such a standard graded polynomial ring over a field $\kappa$ and $I$ be a homogeneous ideal in $R$. It is well-known that the regularity $\reg(R/I^t)$, as a function in $t$, is asymptotically linear for $t\gg 0$ (cf.~\cites{MR1711319, MR1621961}). In general, it is very difficult to decide when this function starts to be linear. To find the exact form of the linear function is also not easy (cf.~\cites{MR3922638, MR2869107, MR2770578, MR4160984}). In the following, we will call them the \emph{linearization-of-regularity problems}, by abuse of terminology.

For these problems, the simplest case is when $I$ is a quadratic squarefree monomial ideal. Whence, it can be recognized as the edge ideal of suitable graphs. And researchers have really obtained some attempting results for a few classes of simple graphs (e.g., forest graphs, cycle graphs, bipartite graphs) (cf.~\cites{MR2923179, MR3417259, MR3757147}). In contrast, little is known about the related binomial ideals. Only recently, Jayanthan et al.~in \cite{JKS} provided the exact formulas for the regularity of powers of binomial edge ideals of several simple graphs such as cycle graphs, star graphs and balloon graphs. And Ene et al.~in \cite{ERT} studied the same problems for connected closed graphs.

Notice that Raghavan in \cite{MR1188639} introduced the notion of quadratic sequence, which generalized the $d$-sequence of Huneke in \cite{MR683201}. And the main result of Jayanthan et al.~in \cite{JKS} is an upper bound for the regularity of powers of an ideal generated by a quadratic sequence. Inspired by their work, in this paper, we will provide exact formulas for the regularity of powers of $I$, when $I$ is an almost complete intersection ideal generated by a $d$-sequence (\Cref{cor:thm_3_general}). Our formulas are written in terms of the regularity of this ideal and one colon ideal only. As applications, we will study the linearization-of-regularity problems when the ideal $I$ is either a binomial edge ideal or a parity binomial edge ideal of some undirected simple graph $G$.

Given such a graph $G$, suppose that it has the vertex set $V(G)=[n]\coloneqq \{1,\ldots,n\}$ and the edge set $E(G)$. Let $S=\kappa[x_1,\dots,x_n,y_1,\dots,y_n]$ denote the accompanying standard graded polynomial ring over the field $\kappa$. Then, its \emph{binomial edge ideal}, introduced by Herzog et al.~in \cite{MR2669070} and independently by Ohtani in \cite{MR2782571}, is defined by
\[
    J_G\coloneqq (x_iy_j-x_jy_i: \{i,j\}\in E(G))\subset S.
\]
Binomial edge ideal has interesting applications in algebraic statistics, in the context of conditional independence ideals (see \cite{MR2669070}). Ever since its debut, the binomial edge ideal has been extensively studied. In \cites{EHH, MR3169567, RG}, the authors considered the Cohen-Macaulay property of chordal graphs, closed graphs, cactus graphs, bicyclic graphs and bipartite graphs, etc. Several nice results on Cohen-Macaulay bipartite graphs and block graphs have been obtained (see \cites{BMS, MR3395714}). The study of the regularity of $J_G$ has also attracted a lot of attention in the recent years, due to its algebraic and geometric importance. In \cite[Theorem 1.1]{MR3084125}, Matsuda and Murai proved that for any graph $G$ on $[n]$, $\ell(G)\leq \reg(S/J_G)\leq n-1$, where $\ell(G)$ is the length of a longest induced path in $G$. After that, many exciting results were obtained regarding the $\reg(S/J_G)$, especially when $G$ is a closed graph, a Cohen-Macaulay bipartite graph, a block graph or a chordal graph (see \cites{JNR, JK, MR3310496, MR4248594} for instance).

As mentioned earlier, another type of binomial ideal that we care is the \emph{parity binomial edge ideal} related to the graph $G$.
This ideal was introduced in \cite{MR3514772} as defined by
\[
    \calI_{G}\coloneqq(x_ix_j-y_iy_j :\{i,j\}\in E(G))\subset S.
\]
In contrast to binomial edge ideals, parity binomial edge ideals bear similar but much more subtler combinatorics.
Kumar in \cite{MR4257788} proved that $\reg(S/\calI_G )\geq \max\{\ell(G), \oc(G)\}$, where $\oc(G)$ denotes the length of a longest induced odd cycle in $G$. Further more, if $G$ is a connected non-bipartite graph on $[n]$ such that $G\setminus e$ is a bipartite graph for some edge $e$, and if $G$ is not an odd cycle, then he proved that $\reg(S/\calI_G)\leq n-1$. Meanwhile, he also characterized all graphs whose parity binomial edge ideals have regularity $3$.

Actually, one can also consider the Lov\'asz–Saks–Schrijver ideal $L_G$ (\cite{MR986889}) and the permanental edge ideal $\Pi_{G}$ (\cite{MR3406962}) associated to $G$. However, by \cite[Remark 3.4]{MR4186617}, we can focus only on binomial edge ideals and parity binomial edge ideals.

As one can observe, for both binomial edge ideals and parity binomial edge ideals, results on the regularity of their powers are not abundant. The work in \cite{ERT} and \cite{JKS} are perhaps the only ones along this research line that we can find so far. That is the reason why we want to push forward a little bit. Due to its difficulty, we have to focus on the cases when $J_G$ and $\calI_G$ are almost complete intersections. Luckily, graphs with almost complete intersection (parity) binomial edge ideals have been completely characterized by the nice work of Jayanthan et al.~in \cite{JKS2} and Kumar in \cite{MR4186617}.
Therefore, as the application of our main results, we can solve the linearization-of-regularity problems for those ideals when the underlying graph $G$ is connected, with only one exception.

This paper is organized as follows. In the next section, we recall several definitions and terminology which we need later. In Sections $3$, we provide the exact formulas for the regularity of powers of an almost complete intersection ideal generated by a homogeneous $d$-sequence, in terms of the regularity of this ideal and a related colon ideal. As applications, in Sections $4$, we compute the regularity of powers of the binomial edge ideals of all connected graphs when the ideals are almost complete intersection. In Sections $5$, we give similar results for parity binomial edge ideals, except when the graph is obtained by adding an edge between two disjoint odd cycles.

\section{Preliminaries}
In this section, we gather together needed definitions and basic facts, which will be used throughout this paper.

\subsection{Basics for simple graphs}

Let $G$ be a simple graph with the vertex set $V(G)=[n]$ and the edge set $E(G)$. It is said to be \emph{bipartite} if there is a bipartition of $V(G)=V_1\sqcup V_2$ such that for each $i=1,2$, no two of the vertices of $V_i$ are adjacent in $G$. Otherwise, it is called a \emph{non-bipartite} graph.  The graph $G$ is called a \emph{complete graph}, if $\{i,j\} \in E(G)$ for all distinct $i,j \in [n]$. A complete graph with $n$ vertices is usually denoted by $K_n$. For any nonempty subset $A\subset V(G)$, $G[A]$ denotes the \emph{induced subgraph} of $G$ on the vertex set $A$, i.e., for $i,j \in A$, $\{i,j\} \in E(G[A])$ if and only if $\{i,j\}\in E(G)$. A subset $U$ of $V(G)$ is said to be a \emph{clique} if $G[U]$ is a complete graph. A vertex $v$ is said to be a \emph{free vertex} if it belongs to exactly one maximal clique; otherwise, it is called an \emph{internal vertex}. In the following, let $\iv(G)$ denote the number of internal vertices of $G$.

The \emph{neighbourhood} of a vertex $v$ in $G$ is defined as $N_G(v)\coloneqq \{u \in V(G)\mid \{u,v\}\in E(G)\}$ and its \emph{degree}, denoted by $\deg_{G}(v)$, is $|N_G(v)|$. For a vertex $v$ in $G$, $G\setminus v$ denotes the induced subgraph of $G$ on the vertex set $V(G)\setminus \{v\}$, and $G_v$ denotes the graph on the vertex set $V(G)$ with edge set $E(G_v)=E(G)\cup \{\{u,w\} \mid u,w \in N_G (v)\}$.

A connected graph $G$ is called a \emph{cycle} if $\deg_{G}(v)=2$ for all $v\in V(G)$. A cycle with $n$ vertices is denoted by $C_n$. A connected graph with vertex set $[n]$ and edge set $E(G)=\Set{\{i,i+1\} \mid 1\leq i\leq n-1}$ is said to be a \emph{path}. Such a path is usually denoted by $P_n$. And the two vertices $1$ and $n$ are called its \emph{end points}. A connected  graph is a \emph{tree} if it does not contain a cycle. A graph is called a \emph{unicyclic graph} if it contains only one cycle. The \emph{girth} of a graph $G$ is the length of a shortest cycle in $G$. A unicyclic graph with even (resp.~odd) girth is called an \emph{even} (resp.~\emph{odd}) \emph{unicyclic graph}.

Meanwhile, a vertex $v$ of $G$ is \emph{cut vertex} in $G$ if $G\setminus v$ has more connected components than $G$.
And a \emph{block} of $G$ is a maximal subgraph without a cut vertex. A connected graph is a \emph{cactus} if its blocks are cycles or edges. And the graph $G$ is a \emph{block graph} if every block of $G$ is a complete graph. In other words, a block graph is a chordal graph such that every pair of blocks of $G$ intersects in at most one vertex.

To study the binomial edge ideal of a simple graph $G$, we have to consider several decomposition operations as follows.
If $e$ is an edge in $G$, then $G\setminus e$ is the graph with the vertex set $V(G)$ and the edge set $E(G)\setminus \{e\}$. On the other hand, if $u,v\in V(G)$ with $e=\{u,v\}\notin E(G)$, then $G_e$ denotes a graph on the vertex set $V(G)$ with the edge set $E(G_e)=E(G)\cup \{\{x,y\}\mid x, y \in N_{G}(u) \text{ or } x,y \in N_{G}(v)\}$.

\subsection{Notions from commutative algebra}

For any homogeneous ideal $I$ of the polynomial ring $R=\kappa[z_{1},\dots,z_{n}]$, there exists a \emph{graded minimal free resolution}
\[
    0\rightarrow \bigoplus\limits_{j}R(-j)^{\beta_{p,j}(R/I)}\rightarrow \bigoplus\limits_{j}R(-j)^{\beta_{p-1,j}(R/I)}\rightarrow \cdots\rightarrow \bigoplus\limits_{j}R(-j)^{\beta_{0,j}(R/I)}\rightarrow R/I\rightarrow 0,
\]
where $R(-j)$ is obtained from $R$ by a shift of degree $j$. The number $\beta_{i,j}(R/I)$, the $(i,j)$-th graded \emph{Betti} number of $R/I$, is an invariant of $R/I$ that equals the number of minimal generators of degree $j$ in the $i$-th syzygy module of $R/I$. The \emph{regularity} of $R/I$, denoted by $\reg(R/I)$, is
\[
    \reg(R/I)\coloneqq\max\,\Set{j-i\mid \beta_{i,j}(R/I)\neq 0}.
\]
Meanwhile, the \emph{projective dimension} of $R/I$, denoted by $\pd(R/I)$, is
\[
    \pd(R/I)\coloneqq\max\,\Set{i\mid \beta_{i,j}(R/I)\neq 0}.
\]
These two invariants measure the complexity of the minimal graded free resolution of $R/I$.

The following lemma is useful when dealing with the regularity of an ideal.

\begin{Lemma}
    [{\cite[Lemma 3.1]{MR2643966}}]
    \label{lem5}
    Let $0\rightarrow M \rightarrow N \rightarrow P \rightarrow 0$ be a short exact sequence of finitely generated graded $R$-modules. Then we have the following.
    \begin{enumerate}[a]
        \item \label{lem5-1} If $\reg(M)\ne \reg(P)+1$, then $\reg(N)=\max\Set{\reg(M), \reg(P)}$.
        \item \label{lem5-2} If $\reg(N)\ne\reg(P)$, then $\reg(M)=\max\Set{\reg(N), \reg(P)+ 1}$.
        \item \label{lem5-3} We always have $\reg(P)\le \max\Set{\reg(M)-1,\reg(N)}$ and the equality holds if $\reg(M)\ne \reg(N)$.
    \end{enumerate}
\end{Lemma}

\begin{Definition}
    A homogeneous ideal $I$ of the polynomial ring $R=\kappa[z_{1},\dots,z_{n}]$ is a \emph{complete intersection} if $\mu(I)=\text{ht}(I)$, where $\mu(I)$ denotes the cardinality of a minimal homogeneous generating set of $I$. It is said to be an \emph{almost complete intersection} if $\mu(I)= \text{ht}(I)+1$ and $I_{\mathfrak{p}}$ is complete intersection for all minimal primes $\mathfrak{p}$ of $I$.
\end{Definition}

It is known that for a simple graph $G$, $J_G$ is complete intersection if and only if all connected components of $G$ are paths (see \cite{EHH}) and $\calI_G$ is a complete intersection if and only if all the bipartite connected components of $G$ are paths while all the non-bipartite connected components are odd cycles (see \cite{BMS}).

\begin{Definition}
    Set $u_0=0\in R$. An ordered sequence of elements $u_1,\ldots,u_n$ in $R$ is said to be a \emph{$d$-sequence} if either of the following equivalent conditions hold:
    \begin{enumerate}[a]
        \item $((u_0,u_1,\ldots,u_{i}): u_{i+1}u_j)=((u_0,u_1,\ldots,u_{i}): u_j)$ for all $0\leq i\leq n-1$ and for all $j\geq i+1$;
        \item $((u_0,u_1,\ldots,u_{i}): u_{i+1})\cap (u_1,\ldots,u_{n})=(u_0,u_1,\ldots,u_{i-1})$ for all $0\leq i\leq n-1$.
    \end{enumerate}
\end{Definition}

\begin{Observation}
    \label{obs:1}
    Suppose that $u_1,\dots,u_n$ form a $d$-sequence in $R$ and $U=(u_1,\dots,u_n)$ is the ideal they generate in $R$. Then, we have
    \begin{equation*}
        ((u_0,u_1,\dots,u_{i-1})+U^t):u_i=((u_0, u_1,\dots,u_{i-1}):u_i)+U^{t-1}
    \end{equation*}
    for $t\ge 1$ and $i=1,\dots,n$. To check the aforementioned equality, it suffices to consider the case when $i=1$, since the images of $u_i,\dots,u_n$ in $R/(u_1,\dots,u_{i-1})$ will form a $d$-sequence by \cite[Remarks after Definition 1.1]{MR683201}. But then it is easy: for any $f\in (U^t:u_1)$, one has $fu_1\in U^t\cap (u_1)=u_1U^{t-1}$ by \cite[Theorem 2.1]{MR683201}. Therefore, $f\in (0:u_1)+U^{t-1}$.
\end{Observation}

\section{Almost complete intersection ideal generated by a homogeneous $d$-sequence}

In this section, we will consider the regularity of powers of an equigenerated almost complete intersection ideal $U$ in some standard graded polynomial ring $R=\kappa[x_1,\dots,x_m]$ over a field $\kappa$. When the field $\kappa$ is infinite, by \cite[Proposition 4.10]{JKS2}, we may assume that $U$ is generated by a $d$-sequence $u_1, \dots , u_{n}$ such that $u_1, \dots , u_{n-1}$ is a regular sequence.

\begin{Theorem}
    \label{thm-3}
    Let $R$ be a standard graded polynomial ring over a field $\kappa$ and $u_1,\dots,u_n\in R$ a homogeneous $d$-sequence of degree $\delta\ge 2$ such that $u_1,\dots,u_{n-1}$ form a regular sequence. Set $U=(u_1,\dots,u_n)$ and assume that $\reg(R/U)\le n\delta-n-\delta$. Then, for all $t\ge 2$ and $i=0,1,2,\dots,n-1$, we have
    \[
        \reg\left(\frac{R}{(u_1,\dots,u_i)+U^t}\right)=n\delta-n-2\delta+\delta t.
    \]
    In particular, $\reg(R/U^t)= n\delta-n-2\delta+\delta t $ for $t\ge 2$.
\end{Theorem}

\begin{proof}
    We will prove by induction on $t\ge 2$ with the $t=2$ case verified separately in \Cref{lem:U2}. Thus, in the following, we may assume that $t\ge 3$. And our arguments depend on the repeated investigation of the following short exact sequences
    \begin{equation}
        0\to \frac{R}{((u_1,\dots,u_j)+U^t):u_{j+1}}(-\delta) \xrightarrow{\cdot u_{j+1}} \frac{R}{(u_1,\dots,u_j)+U^t} \to \frac{R}{(u_1,\dots,u_{j+1})+U^t} \to 0
        \label{eqn:SES-1}
    \end{equation}
    for various $j$ and $t$. Since the $d$-sequence $u_1,\dots,u_n$ generates the ideal $U$, we have
    \begin{equation}
        ((u_1,\dots,u_j)+U^t):u_{j+1}=((u_1,\dots,u_j):u_{j+1})+U^{t-1}
        \label{eqn:colon-1}
    \end{equation}
    by \Cref{obs:1}.

    Now, we will prove the statements by the descending induction on $i$. Firstly, we consider the $i=n-1$ case
    by looking at the short exact sequence in \eqref{eqn:SES-1} with $j=n-1$. Observe that $(u_1,\dots,u_{(n-1)+1})+U^t=U$ in this case. Meanwhile,
    \[
        \frac{R}{((u_1,\dots,u_{n-1}):{u_{n}})+{U}^{t-1}}=\frac{R}{((u_1,\dots,u_{n-1}):u_n)+(u_n^{t-1})},
    \]
    which has regularity $n\delta-n-3\delta+\delta t$ by \Cref{lem:1.9}. As $t\ge 3$ and $\delta\ge 2$, we have
    \begin{align*}
        & \reg\left(\frac{R}{((u_1,\dots,u_{n-1}):{u_{n}})+{U}^{t-1}}(-\delta)\right)=\reg\left(\frac{R}{((u_1,\dots,u_{n-1}):{u_{n}})+{U}^{t-1}}\right)+\delta\\
        &= (n\delta-n-3\delta+\delta t)+\delta
        > (n-1)(\delta-1)+1>\reg(R/U)+1
    \end{align*}
    by the assumption on $\reg(R/U)$. Thus, we obtain
    \[
        \reg\left(\frac{R}{(u_1,\dots,u_{n-1})+U^t}\right)=n\delta-n-2\delta+\delta t,
    \]
    by item \ref{lem5-1} of \Cref{lem5}, as claimed.

    Next, we assume that $i\le n-2$ and the assertion holds for $i+1$. Once again, we look at the short exact sequence \eqref{eqn:SES-1} with $j=i$. Notice that ${R}/{((u_1,\dots,u_{i+1})+U^t)}$ has regularity $n\delta-n-2\delta+\delta t$ by induction on $i+1$. Meanwhile, since $u_1,\dots,u_{n-1}$ form a regular sequence, it follows from the equation \eqref{eqn:colon-1} that
    \[
        \frac{R}{(((u_1,\dots,u_i)+U^t):u_{i+1})}= \frac{R}{(u_1,\dots,u_i)+U^{t-1}},
    \]
    which has regularity $n\delta-n-3\delta+\delta t$ by induction on $t$. As $(n\delta-n-3\delta+\delta t)+\delta<(n\delta-n-2\delta+\delta t)+1$, again by item \ref{lem5-1} of \Cref{lem5}, ${R}/((u_1,\dots,u_i)+U^t)$ has regularity $n\delta-n-2\delta+\delta t$, as expected. And this completes our proof.
\end{proof}

To fully complete the proof of \Cref{thm-3}, we still have three small results to show. The first one deals with the $t=2$ case.

\begin{Lemma}
    \label{lem:U2}
    Under the assumptions in \Cref{thm-3}, we have
    \[
        \reg\left(\frac{R}{(u_1,\dots,u_i)+U^2}\right)=n\delta-n
    \]
    for $i=0,1,\dots, n-1$. 
\end{Lemma}

\begin{proof}
    First, we consider the short exact sequence
    \begin{equation}
        0\to \frac{R}{(u_1,\dots,u_{n-1}):u_n}(-\delta)\xrightarrow{\cdot u_n} \frac{R}{(u_1,\dots,u_{n-1})}\to \frac{R}{U}\to 0.
        \label{eqn:a}
    \end{equation}
    Since $u_1,\dots,u_{n-1}$ form a regular sequence, we can read from its associated Koszul complex that $\reg(R/(u_1,\dots,u_{n-1}))=(n-1)(\delta-1)>\reg(R/U)$. Therefore, it follows from item \ref{lem5-2} of \Cref{lem5} that
    \[
        \reg\left(\frac{R}{(u_1,\dots,u_{n-1}):u_n}\right)=(n-1)(\delta-1)-\delta=n\delta-n-2\delta+1.
    \]

    In the following, we will prove the statements by a descending induction on $i$. For $i=n-1$, we look at the short exact sequence
    \[
        0\to \frac{R}{(u_1,\dots,u_{n-1}):u_n^2}(-2\delta)\xrightarrow{\cdot u_n^2} \frac{R}{(u_1,\dots,u_{n-1})}\to \frac{R}{(u_1,\dots,u_{n-1},u_n^2)}\to 0.
    \]
    Since $u_1,\dots,u_n$ form a $d$-sequence, we have $(u_1,\dots,u_{n-1}):u_n^2=(u_1,\dots,u_{n-1}):u_n$. As
    \[
        \reg\left(\frac{R}{(u_1,\dots,u_{n-1}):u_n^2}(-2\delta)\right)=n\delta-n+1>\reg\left(\frac{R}{(u_1,\dots,u_{n-1})}\right),
    \]
    it follows from item \ref{lem5-3} of \Cref{lem5} that
    \[
        \reg\left(\frac{R}{(u_1,\dots,u_{n-1})+U^2}\right)=\reg\left(\frac{R}{(u_1,\dots,u_{n-1})+u_n^2}\right)=n\delta-n.
    \]

    Next, we assume that $i\le n-2$ and the assertion holds for $i$. Thus, we turn to the short exact sequence
    \[
        0\to \frac{R}{((u_1,\dots,u_{i-1})+U^2):u_{i}}(-\delta)\xrightarrow{\cdot u_i} \frac{R}{(u_1,\dots,u_{i-1})+U^2}\to \frac{R}{(u_1,\dots,u_i)+U^2} \to 0.
    \]
    Since $u_1,\dots,u_{n-1}$ form a regular sequence and $u_1,\dots,u_n$ form a $d$-sequence, we obtain from \Cref{obs:1} that
    \[
        ((u_1,\dots,u_{i-1})+U^2):u_{i}=((u_1,\dots,u_{i-1}):u_i)+U^{2-1}=(u_1,\dots,u_{i-1})+U=U.
    \]
    By the inductive hypothesis that $\reg({R}/((u_1,\dots,u_i)+U^2))=n\delta-n$ and the assumption that $\reg(R/U)\le n\delta-n-\delta$, we obtain that
    \begin{align*}
        \reg\left(\frac{R}{((u_1,\dots,u_{i-1})+U^2):u_{i}}(-\delta)\right)
        =\reg\left(\frac{R}{U}\right)+\delta <\reg\left(\frac{R}{(u_1,\dots,u_i)+U^2}\right)+1.
    \end{align*}
    Thus, by applying item \ref{lem5-1} of \Cref{lem5}, we can get
    \[
        \reg\left(\frac{R}{(u_1,\dots,u_{i-1})+U^2}\right)= \reg\left(\frac{R}{(u_1,\dots,u_i)+U^2}\right)=n\delta-n,
    \]
    as claimed.
\end{proof}

\begin{Lemma}
    \label{lem:colon_n_minus_1}
    Let $R$ be a standard graded polynomial ring over a field $\kappa$ and $u_1,\dots,u_n\in R$ a homogeneous $d$-sequence of degree $\delta\ge 2$ such that $u_1,\dots,u_{n-1}$ form a regular sequence. Then, for $t\ge 1$, we have
    \[
        \reg\left(\frac{R}{((u_1,\dots,u_{n-1}):u_n)+(u_n^t)}\right)=\reg\left(\frac{R}{(u_1,\dots,u_{n-1}):u_n}\right)+\delta t-1.
    \]
\end{Lemma}

\begin{proof}
    We consider the following exact sequence of graded modules:
    \begin{equation*}
        0\to \frac{R}{(u_1,\dots,u_{n-1}):u_n^{t+1}}(-\delta t)\xrightarrow{\cdot u_n^t} \frac{R}{(u_1,\dots,u_{n-1}):u_n} \to \frac{R}{((u_1,\dots,u_{n-1}):u_n)+(u_n^t)}\to 0.
    \end{equation*}
    Since $u_1,\dots,u_n$ form a $d$-sequence, one has $(u_1,\dots,u_{n-1}):u_n^{t+1}=(u_1,\dots,u_{n-1}):u_n$. By applying \ref{lem5-3} of \Cref{lem5} to the exact sequence above, we obtain the desired equality.
\end{proof}

Here is the last piece that we need for \Cref{thm-3}.

\begin{Lemma}
    \label{lem:1.9}
    Under the assumptions in \Cref{thm-3}, we have
    \[
        \reg\left(\frac{R}{((u_1,\dots,u_{n-1}):u_n)+(u_n^t)}\right)=n\delta-n-2\delta+\delta t
    \]
    for $t\ge 1$.
\end{Lemma}

\begin{proof}
    Let us look at the short exact sequence of graded modules:
    \begin{equation*}
        0\to \frac{R}{(u_1,\dots,u_{n-1}):u_n}(-\delta)\xrightarrow{\cdot u_n} \frac{R}{(u_1,\dots,u_{n-1})}\to \frac{R}{(u_1,\dots,u_n)}\to 0.
    \end{equation*}
    Since the sequence $u_1, \dots,u_{n-1}$ is a regular sequence and each $\deg(u_i)=\delta$, we have
    \[
        \reg\left(\frac{R}{(u_1,\dots,u_{n-1})} \right)=(n-1)(\delta-1).
    \]
    Meanwhile, since $\reg(R/U)< (n-1)(\delta-1)$ by assumption, it follows from the item \ref{lem5-2} of \Cref{lem5} that
    \[
        \reg\left(\frac{R}{(u_1,\dots,u_{n-1}):u_n} \right)+\delta=\reg\left(\frac{R}{(u_1,\dots,u_{n-1})} \right),
    \]
    namely,
    \[
        \reg\left(\frac{R}{(u_1,\dots,u_{n-1}):u_n} \right)=n\delta-n-2\delta+1.
    \]
    After this, we can apply \Cref{lem:colon_n_minus_1}.
\end{proof}

Therefore, we complete the proof of \Cref{thm-3}. And by applying a similar technique, we can get the following result.

\begin{Theorem}
    \label{thm:HT}
    Let $R$ be a standard graded polynomial ring over a field $\kappa$ and $u_1,\dots,u_n\in R$ a homogeneous $d$-sequence of degree $\delta\ge 2$ such that $u_1,\dots,u_{n-1}$ form a regular sequence. Set $U=(u_1,\dots,u_n)$ and write $\reg(R/((u_1,\dots,u_{n-1}):u_n))=B$. If
    \[
        B\ge \max\Set{\reg\left(\frac{R}{U}\right)-\delta+1,n\delta-n-3\delta+2},
    \]
    then, for all $t\ge 2$ and $i=0,1,2,\dots,n-1$, we have
    \[
        \reg\left(\frac{R}{(u_1,\dots,u_i)+U^t}\right)=B+\delta t-1.
    \]
    In particular, $\reg(R/U^t)= B+\delta t-1$ for $t\ge 2$.
\end{Theorem}

\begin{proof}
    We will prove the assertions by induction on $t$. The $t=2$ case will be shown separately in \Cref{lem:HT_U2}. Thus, we may assume that $t\ge 3$, and prove the statements by descending induction on $i$. For $i=n-1$, we note that
    \[
        ((u_1,\dots,u_{n-1})+U^t):u_{n}=((u_1,\dots,u_{n-1}):u_{n})+U^{t-1}=((u_1,\dots,u_{n-1}):u_n)+(u_n^{t-1})
    \]
    by \Cref{obs:1}, and $(u_1,\dots,u_{n})+U^t=U$. Furthermore, by \Cref{lem:colon_n_minus_1}, we have
    \[
        \reg\left(\frac{R}{((u_1,\dots,u_{n-1}):u_n)+(u_n^{t-1})}\right)=B+\delta t-\delta-1.
    \]
    Since $B\ge \reg(R/U)-\delta+1$ while $t\ge 3$, after applying \ref{lem5-1} of \Cref{lem5} to the exact sequence \eqref{eqn:SES-1} with $j=n-1$, we obtain
    \[
        \reg\left(\frac{R}{(u_1,\dots,u_{n-1})+U^t}\right)=B+\delta t-1.
    \]

    Next, we assume that $i\le n-1$ and the assertion holds for $i$.
    Note that
    \begin{align*}
        \reg\left(\frac{R}{((u_1,\dots,u_{i-1})+U^t):u_{i}}(-\delta )\right)
        =\reg\left(\frac{R}{(u_1,\dots,u_{i-1})+U^{t-1}}\right)+\delta
        =B+\delta t-1
    \end{align*}
    by the induction on $t$ and the assumption that $\reg({R}/({(u_1,\dots,u_{i})+U^t}))=B+\delta t-1$. Applying again item \ref{lem5-1} of \Cref{lem5} to the exact sequence \eqref{eqn:SES-1} with $j=i-1$, one has
    \[
        \reg\left(\frac{R}{(u_1,\dots,u_{i-1})+U^t}\right)=B+\delta t-1.
    \]
    And this completes the proof.
\end{proof}

Below is the $t=2$ case for \Cref{thm:HT}, treated separately.

\begin{Lemma}
    \label{lem:HT_U2}
    Under the assumptions in \Cref{thm:HT}, we have
    \[
        \reg\left(\frac{R}{(u_1,\dots,u_i)+U^2}\right)=B+2\delta-1
    \]
    for $i=0,1,\dots, n-1$. 
\end{Lemma}

\begin{proof}
    %
    We will prove the statements by descending induction on $i$. For $i=n-1$, we look at the short exact sequence
    \[
        0\to \frac{R}{(u_1,\dots,u_{n-1}):u_n^2}(-2\delta)\xrightarrow{\cdot u_n^2} \frac{R}{(u_1,\dots,u_{n-1})}\to \frac{R}{(u_1,\dots,u_{n-1},u_n^2)}\to 0.
    \]
    Since $u_1,\dots,u_n$ form a $d$-sequence, we have $(u_1,\dots,u_{n-1}):u_n^2=(u_1,\dots,u_{n-1}):u_n$. As
    \[
        \reg\left(\frac{R}{(u_1,\dots,u_{n-1}):u_n^2}(-2\delta)\right)=B+2\delta>\reg\left(\frac{R}{(u_1,\dots,u_{n-1})}\right)=(n-1)(\delta-1)
    \]
    by our assumption on $B$, it follows from item \ref{lem5-3} of \Cref{lem5} that
    \[
        \reg\left(\frac{R}{(u_1,\dots,u_{n-1})+U^2}\right)=\reg\left(\frac{R}{(u_1,\dots,u_{n-1},u_n^2)}\right)=B+2\delta-1.
    \]

    Next, we assume that $i\le n-2$ and the statement holds for $i$. Thus, we turn to the short exact sequence
    \[
        0\to \frac{R}{((u_1,\dots,u_{i-1})+U^2):u_{i}}(-\delta)\xrightarrow{\cdot u_i} \frac{R}{(u_1,\dots,u_{i-1})+U^2}\to \frac{R}{(u_1,\dots,u_i)+U^2} \to 0.
    \]
    Since $u_1,\dots,u_{n-1}$ form a regular sequence and $u_1,\dots,u_n$ form a $d$-sequence, we obtain from \Cref{obs:1} that
    \[
        ((u_1,\dots,u_{i-1})+U^2):u_{i}=((u_1,\dots,u_{i-1}):u_i)+U^{2-1}=(u_1,\dots,u_{i-1})+U=U.
    \]
    By
    the inductive hypothesis that $\reg({R}/((u_1,\dots,u_{i})+U^2)) =B+2\delta -1$ and the assumption that $\reg(R/U)\le B+\delta-1$, we have
    \[
        \reg\left(\frac{R}{((u_1,\dots,u_{i-1})+U^2):u_{i}}(-\delta)\right)
        <\reg\left(\frac{R}{(u_1,\dots,u_i)+U^2}\right)+1.
    \]
    Thus, by applying item \ref{lem5-1} of \Cref{lem5}, we can get
    \[
        \reg\left(\frac{R}{(u_1,\dots,u_{i-1})+U^2}\right)= \reg\left(\frac{R}{(u_1,\dots,u_i)+U^2}\right)=B+2\delta-1,
    \]
    as claimed.
\end{proof}

\begin{Corollary}
    \label{cor:HT}
    Let $R$ be a standard graded polynomial ring over a field $\kappa$ and $u_1,\dots,u_n\in R$ a homogeneous $d$-sequence of degree $\delta\ge 2$ such that $u_1,\dots,u_{n-1}$ form a regular sequence. Set $U=(u_1,\dots,u_n)$ and assume that $\reg({R}/{U})\ne (n-1)(\delta-1)$. Then, for all $t\ge 2$, we have
    \[
        \reg\left(\frac{R}{U^t}\right)=\reg\left(\frac{R}{(u_1,\dots,u_{n-1}):u_{n}}\right)+\delta t-1.
    \]
    In particular, if $\reg({R}/{U})\ge n\delta-n-\delta+2$ or $\reg({R}/{U})=n\delta-n-\delta$, then, for all $t\ge 1$, we have
    \[
        \reg\left(\frac{R}{U^t}\right)=\reg\left(\frac{R}{U}\right)+\delta t-\delta.
    \]
\end{Corollary}

\begin{proof}
    Since $u_1,\dots,u_{n-1}$ form a regular sequence, one has $\reg(R/(u_1,\dots,u_{n-1}))=(n-1)(\delta-1)$. Applying item \ref{lem5-2} of \Cref{lem5} to the short exact sequence \eqref{eqn:a}, we have
    \begin{equation}
        \reg\left(\frac{R}{(u_1,\dots,u_{n-1}):u_{n}}\right)+\delta=\max\Set{ \reg\left(\frac{R}{(u_1,\dots,u_{n-1})}\right),\reg\left(\frac{R}{U}\right)+1}.
        \label{eqn:colon_un}
    \end{equation}
    Since $\delta\ge 2$, this implies that
    \[
        \reg\left(\frac{R}{(u_1,\dots,u_{n-1}):u_{n}}\right)\ge \max\Set{n\delta-n-3\delta+2,\reg\left(\frac{R}{U}\right)-\delta+1}.
    \]
    Hence, for all $t\ge 2$, it follows from \Cref{thm:HT} that
    \begin{equation}
        \reg\left(\frac{R}{U^t}\right)=\reg\left(\frac{R}{(u_1,\dots,u_{n-1}):u_{n}}\right)+\delta t-1.
        \label{eqn:b}
    \end{equation}	

    In particular, if $\reg({R}/{U})\ge n\delta-n-\delta+2$ or $\reg({R}/{U})=n\delta-n-\delta$, then obviously $\reg({R}/{U})\ne (n-1)(\delta-1)$. Meanwhile, we will have
    \[	
        \reg\left(\frac{R}{(u_1,\dots,u_{n-1}):u_{n}}\right)=\reg\left(\frac{R}{U}\right)-\delta+1
    \]
    from the equation \eqref{eqn:colon_un}. It remains to substitute this into the equation \eqref{eqn:b} to obtain the desired results.
\end{proof}

\begin{Corollary}
    \label{cor:thm_3_general}
    \label{cor:HT_general}
    Let $R$ be a standard graded polynomial ring over an infinite field $\kappa$. Suppose that $U\subset R$ is an equigenerated almost complete intersection ideal by some forms of degree $\delta\ge 2$ with $\Ht(U)=n-1$.
    \begin{enumerate}[a]
        \item If $\reg(R/U)\le n\delta-n-\delta$, then
            \[
                \reg\left(\frac{R}{U^t}\right)=n\delta-n+\delta t-2\delta
            \]
            for all $t\ge 2$.
        \item If $\reg({R}/{U})\ge n\delta-n-\delta+2$, then
            \[
                \reg\left(\frac{R}{U^t}\right)=\reg\left(\frac{R}{U}\right)+\delta t-\delta
            \]
            for all $t\ge 1$.
    \end{enumerate}
\end{Corollary}

\begin{proof}
    It follows from \cite[Proposition 4.10]{JKS2} that there exists a  system of homogeneous generators $\{u_1, \dots, u_{n}\}$ of $U$ such that $u_1, \dots , u_{n-1}$ is a regular sequence while $u_1, \dots , u_{n}$ is a $d$-sequence.
    The desired results follow from \Cref{thm-3} and \Cref{cor:HT} respectively now.
\end{proof}

\section{Graphs with almost complete intersection binomial edge ideals}

In this section, we will study the regularity of powers of the binomial edge ideal of a graph $G$ whose binomial edge ideal $J_G$ is an almost complete intersection. Without loss of generality, suppose that the vertex set of $G$ is $[n]$. Let $S=S_G=\kappa[x_i,y_i\mid 1\le i\le n]$ be a standard graded polynomial ring in $2n$ variables over an infinite field $\kappa$. And for simplicity, if $e=\{i,j\}$ is an edge of a graph $G$, we write ${f}_e={f}_{i,j}\coloneqq x_iy_j-x_jy_i$. Consequently, the binomial edge ideal $J_G=(f_e\mid e\in E(G))$ in $S$.
It is well-known that for a simple graph $G$, $J_G$ is complete intersection if and only if all connected components of $G$ are paths (see \cite{EHH}).

Let $G$ be a graph and $v$ be a \emph{cut vertex} in $G$. By definition, this simply means $G\setminus v$ has more connected components than $G$.  Let $G_1,\ldots,G_k$ be the components of $G\setminus \{v\}$ and $G'_i=G[V(G_i)\cup \{v\}]$, the subgraph of $G$ induced by $V(G_i)\cup \{v\}$. Then, $G'_1,\ldots, G'_k$ form the \emph{split of $G$ at $v$}. 
Meanwhile, recall that a vertex $v$ of $G$ is said to be a \emph{free vertex} if $v$ is contained in only one maximal clique; otherwise it is called an \emph{internal vertex}.

\begin{Lemma}
    \label{lem:reduction_by_path}
    Suppose that the induced subgraph $G'$ and path $P_n$ form a split of a graph $G$ at the vertex $v$. If $v$ is the end vertex of $P_n$, and is also a free vertex of $G'$, then $\reg(S/J_G) = \reg(S/J_{G'}) + n-1$.
\end{Lemma}

\begin{proof}
    It follows from \cite[Theorem 3.1]{JNR} that $\reg(S/J_G) = \reg(S/J_{G'}) + \reg(S/P_n)$. But since the quadratic ideal $J_{P_n}$ is a complete intersection, one has $\reg(S/J_{P_n})=n-1$.
\end{proof}

Recall that Jayanthan et al.~classified all connected graphs whose binomial edge ideals are almost complete intersection. 

\begin{Lemma}
    [{\cite[Theorems 4.3 and 4.4]{JKS2}}]
    \label{lem6}
    Let $G$ be a connected graph.
    \begin{enumerate}[i]
        \item If $G$ is not a tree, then $J_G$ is an almost complete intersection ideal if and only if $G$ is obtained by adding an edge between two vertices of a path or by attaching a path to each vertex of a $3$-cycle $C_3$.
        \item If $G$ is a tree but not a path, then $J_G$ is an almost complete intersection ideal if and only if $G$ is obtained by adding an edge between two vertices of two paths.
    \end{enumerate}	
\end{Lemma}

Then, the following fact can be observed with ease.

\begin{Corollary}
    Let $G=G_1\sqcup \cdots\sqcup G_k$ be a disjoint union of $k$ graphs. Then $J_G$ is almost complete intersection if and only if for some $i$, $J_{G_{i}}$ is almost complete intersection and for each $j\neq i$, $J_{G_{j}}$ is complete intersection.
\end{Corollary}

Notice that if $\Char(\kappa)=0$, then all regularity computations of powers of ideals can be reduced to the case when $G$ is connected, by \cite[Lemma 4.1, Proposition 5.1]{MR3912960}. Therefore, in the following, we will only consider the case when $G$ is a connected graph in \Cref{lem6}.

\subsection{$G_2$-type cases}
Earlier in \cite{JKS}, two graphs on some vertex set $[m]$, called $G_1$ and $G_2$ respectively, were considered; see also the graphs displayed in \Cref{Fig:3}. Related, in this subsection, let $G$ is a graph on $[n]$ obtained by adding an edge between two vertices $i_0$ and $j_0$ of a path $P_n$ as in \Cref{lem6}. If the two vertices $i_0$ and $j_0$ are precisely the two end points of the path, then we obtain a cycle $C_n$. This case has already been studied in \cite[Theorem 3.6]{JKS}. If precisely one such vertex, say $i_0$, is an end vertex of $P_n$, then $G$ is obtained by identifying an end vertex of a path with the pendant vertex of $G_1$ (the vertex $1$ of $G_1$ in \Cref{Fig:3}). Whence, we will call it of \emph{$G_1$-type}. We can also call it a \emph{balloon graph} for the obvious reason. And this case has also been considered in \cite[Remark 3.15]{JKS}. Therefore, it remains to deal with the case when neither $i_0$ nor $j_0$ is an end vertex of $P_n$. Whence, $G$ is obtained by identifying each of the pendant vertices of $G_2$ with an end vertex of a new path respectively. And we will call it of \emph{$G_2$-type}.

\begin{figure}[htbp]%
    \centering
    \parbox[b]{.35\textwidth}{\centering
        \begin{tikzpicture}[thick, scale=1.5, every node/.style={scale=1.2}]]
            \foreach \x in {1,...,10} {
                \shade [shading=ball, ball color=black] (\x*36:1) circle (.07);
                \draw (\x*36:1) -- (\x*36+36:1);
            }
            \node at (3*36:1) [above left] {\scriptsize$2$};
            \node at (2*36:1) [above right] {\scriptsize$m$};
            \draw [thick] (3*36:1) -- +(-0.1,0.6) node [above left] {\scriptsize $1$};
            \shade [shading=ball, ball color=black] (3*36:1)+(-0.1,0.6) circle (.07);
        \end{tikzpicture}
        \subcaption*{$G_1$}
    } \qquad
    \parbox[b]{.35\textwidth}{\centering
        \begin{tikzpicture}[thick, scale=1.5, every node/.style={scale=1.2}]]
            \foreach \x in {1,...,10} {
                \shade [shading=ball, ball color=black] (\x*36:1) circle (.07);
                \draw (\x*36:1) -- (\x*36+36:1);
            }
            \node at (3*36:1) [above left] {\scriptsize$2$};
            \node at (2*36:1) [above right] {\scriptsize$m-1$};
            \draw [thick] (2*36:1) -- +(0.1,0.6) node [above right] {\scriptsize $m$};
            \draw [thick] (3*36:1) -- +(-0.1,0.6) node [above left] {\scriptsize $1$};
            \shade [shading=ball, ball color=black] (2*36:1)+(0.1,0.6) circle (.07);
            \shade [shading=ball, ball color=black] (3*36:1)+(-0.1,0.6) circle (.07);
        \end{tikzpicture}
        \subcaption*{$G_2$}
    }
    \caption{\ }
    \label{Fig:3}
\end{figure}

\begin{Theorem}
    \label{thm:F3-reg}
    Let $G$ be a connected graph on $[n]$ which is obtained by adding an edge between two vertices of a path $P_{n}$. If the girth of $G$ is at least $4$, then $\reg(S/J_{G}^t)= 2t+n-4$ for all $t\ge 2$.
\end{Theorem}

\begin{proof}
    As mentioned earlier, it suffices to consider the case when $G$ is of $G_2$-type. We claim that in this case $\reg(S/J_G)=n-3$. Whence, we can apply \Cref{cor:thm_3_general} to achieve the desired result.

    To confirm the claim, by \Cref{lem:reduction_by_path}, it suffices to assume that
    $G=G_2$ with $n=m$. But then, the claimed formula has been proved by \cite[Proposition 3.13]{JKS}, which completes the proof.
    \qedhere

\end{proof}

Interestingly, Bolognini et al.~in \cite{BMS} studied a family of bipartite graphs, denoted by $F_m$, whose binomial edge ideals are Cohen-Macaulay. Let $m$ be a positive integer. Then, $F_m$ is the graph on the vertex set $[2m]$ with the edge set $E(F_m)=\{\{2i,2j-1\}: 1\le i\le j\le m\}$. The graphs $F_3$ and $F_4$ are displayed in \Cref{Fig:1}.

\begin{figure}[htbp]%
    \centering
    \parbox[b]{.3\textwidth}{\centering
        \begin{tikzpicture}[thick, scale=1.2, every node/.style={scale=0.8}]]
            \shade [shading=ball, ball color=black]  (1,1) circle (.07) node [below] {\scriptsize$1$};
            \shade [shading=ball, ball color=black]  (2,1) circle (.07) node [below] {\scriptsize$3$};
            \shade [shading=ball, ball color=black]  (3,1) circle (.07) node [below] {\scriptsize$5$};
            \shade [shading=ball, ball color=black]  (1,3) circle (.07) node [above] {\scriptsize$2$};
            \shade [shading=ball, ball color=black]  (2,3) circle (.07) node [above] {\scriptsize$4$};
            \shade [shading=ball, ball color=black]  (3,3) circle (.07) node [above] {\scriptsize$6$};
            \draw[thick] (1,1) -- (1,3);
            \draw[thick] (1,1) -- (2,3);
            \draw[thick] (1,1) -- (3,3);
            \draw[thick] (2,1) -- (2,3);
            \draw[thick] (2,1) -- (3,3);
            \draw[thick] (3,1) -- (3,3);
        \end{tikzpicture}
        \subcaption*{The graph $F_3$}
    }\qquad
    \parbox[b]{.3\textwidth}{\centering
        \begin{tikzpicture} [thick, scale=1.0, every node/.style={scale=0.98}]]
            \shade [shading=ball, ball color=black]  (2,1) circle (.07) node [below] {\scriptsize$4$};
            \shade [shading=ball, ball color=black]  (3.5,1) circle (.07) node [below] {\scriptsize$3$};
            \shade [shading=ball, ball color=black]  (2,2.5) circle (.07) node [above] {\scriptsize$1$};
            \shade [shading=ball, ball color=black]  (3.5,2.5) circle (.07) node [above] {\scriptsize$6$};
            \shade [shading=ball, ball color=black]  (1,3.2) circle (.07) node [above] {\scriptsize$2$};
            \shade [shading=ball, ball color=black]  (4.5,3.2) circle (.07) node [above] {\scriptsize$5$};
            \draw[thick] (2,1) -- (3.5,1);
            \draw[thick] (2,2.5) -- (3.5,2.5);
            \draw[thick] (2,1) -- (2,2.5);
            \draw[thick] (1,3.2) -- (2,2.5);
            \draw[thick] (3.5,1) -- (3.5,2.5);
            \draw[thick] (4.5,3.2) -- (3.5,2.5);
        \end{tikzpicture}
        \subcaption*{The graph $F_3$ is of $G_2$-type}
    }\qquad
    \parbox[b]{.3\textwidth}{\centering
        \begin{tikzpicture} [thick, scale=1.2, every node/.style={scale=0.8}]]
            \shade [shading=ball, ball color=black]  (1,1) circle (.07) node [below] {\scriptsize$1$};
            \shade [shading=ball, ball color=black]  (2,1) circle (.07) node [below] {\scriptsize$3$};
            \shade [shading=ball, ball color=black]  (3,1) circle (.07) node [below] {\scriptsize$5$};
            \shade [shading=ball, ball color=black]  (4,1) circle (.07) node [below] {\scriptsize$7$};
            \shade [shading=ball, ball color=black]  (1,3) circle (.07) node [above] {\scriptsize$2$};
            \shade [shading=ball, ball color=black]  (2,3) circle (.07) node [above] {\scriptsize$4$};
            \shade [shading=ball, ball color=black]  (3,3) circle (.07) node [above] {\scriptsize$6$};
            \shade [shading=ball, ball color=black]  (4,3) circle (.07) node [above] {\scriptsize$8$};
            \draw[thick] (1,1) -- (1,3);
            \draw[thick] (1,1) -- (2,3);
            \draw[thick] (1,1) -- (3,3);
            \draw[thick] (1,1) -- (4,3);
            \draw[thick] (2,1) -- (2,3);
            \draw[thick] (2,1) -- (3,3);
            \draw[thick] (2,1) -- (4,3);
            \draw[thick] (3,1) -- (3,3);
            \draw[thick] (3,1) -- (4,3);
            \draw[thick] (4,1) -- (4,3);
        \end{tikzpicture}
        \subcaption*{The graph $F_4$}
    }
    \caption{\ }
    \label{Fig:1}
\end{figure}

\begin{Corollary}
    Let $m\ge 3$ be a positive integer, then $\reg(S/J_{F_m}^t)\ge 2t+2$ for all $t\ge 2$.
\end{Corollary}

\begin{proof}
    Note that $F_3$ is an induced subgraph of $F_m$. Thus, by \cite[Proposition 3.3]{JKS}, we have $\reg(S/J_{F_m}^t)\ge \reg(S/J_{F_3}^t)$ for all $t\ge 2$. Meanwhile, $F_3$ is the $G_2$ with $6$ vertices. Hence, the assertion immediately follows from the theorem above.
\end{proof}

\subsection{$C_3$-type cases}

In this subsection, let $G$ be a graph which is obtained by attaching a path to each vertex of a $3$-cycle $C_3$ as in \Cref{lem6}; see for instance the graph displayed in \Cref{Fig:4}. Whence, we will call it of \emph{$C_3$-type}. Here, we allow the degenerated case when some of the paths attached to be empty.

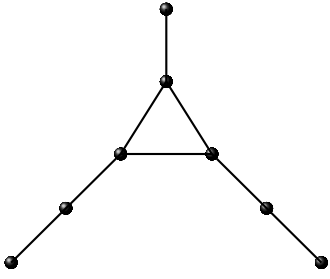
\begin{figure}[htbp]%
    \centering
    \begin{tikzpicture}[thick, scale=1.2, every node/.style={scale=0.8}]]
        \shade [shading=ball, ball color=black]  (1.8,1.8) circle (.07);
        \shade [shading=ball, ball color=black]  (2.4,2.4) circle (.07);
        \shade [shading=ball, ball color=black]  (3,3) circle (.07);
        \shade [shading=ball, ball color=black]  (3.5,3.8) circle (.07);
        \shade [shading=ball, ball color=black]  (3.5,4.6) circle (.07);
        \shade [shading=ball, ball color=black]  (3.5,3.8) circle (.07);
        \shade [shading=ball, ball color=black]  (4,3) circle (.07);
        \shade [shading=ball, ball color=black]  (4.6,2.4) circle (.07);
        \shade [shading=ball, ball color=black]  (5.2,1.8) circle (.07);
        \draw[thick] (1.8,1.8) -- (2.4,2.4) -- (3,3) -- (3.5,3.8) -- (3.5,4.6);
        \draw[thick] (3.5,3.8) -- (4,3) -- (4.6,2.4) -- (5.2,1.8);
        \draw[thick] (3,3) -- (4,3);
    \end{tikzpicture}
    \caption{$C_3$-type graph}
    \label{Fig:4}
\end{figure}


\begin{Theorem}
    \label{thm:C3}
    Let $G$ be a $C_3$-type graph on $[n]$. Then, we have $\reg(S/J_G^t)=2t+n-4$ for $t\ge 1$.
\end{Theorem}

\begin{proof}
    We claim that $\reg(S/J_G)=n-2$. With this, since $\Ht(J_G)=n-1$, it then suffices to apply the first case of \Cref{cor:HT_general}.

    As for the claim, by \Cref{lem:reduction_by_path}, it suffices to assume that $n=3$ and $G=C_3$. But $\reg(S/J_{C_3})=1$ is already known by \cite[Theorem 3.6]{JKS}, which completes the proof. 
\end{proof}

\subsection{Tree cases}

In this subsection, let the graph $G$ be obtained by adding an edge between two vertices of two paths as in \Cref{lem6}. If $G$ is obtained by adding an edge between an internal vertex of a path and an end vertex of another path, then we say that $G$ is of \emph{$T$-type}. And if $G$ is obtained by adding an edge between two internal vertices of two distinct paths, then we say that $G$ is of \emph{$H$-type}. Typical examples of $T$-type graph and $H$-type graphs are displayed in \Cref{Fig:2}.

\begin{figure}[htbp]%
    \centering
    \parbox[b]{.25\textwidth}{\centering
        \begin{tikzpicture}[thick, scale=1.5, every node/.style={scale=0.8}]]
            \shade [shading=ball, ball color=black]  (1,1) circle (.06);
            \shade [shading=ball, ball color=black]  (2,1) circle (.06);
            \shade [shading=ball, ball color=black]  (3,1) circle (.06);
            \shade [shading=ball, ball color=black]  (4,1) circle (.06);
            \shade [shading=ball, ball color=black]  (2,0) circle (.06);
            \shade [shading=ball, ball color=black]  (2,-1) circle(.06);
            \draw[thick] (1,1) -- (2,1) -- (3,1) -- (4,1);
            \draw[thick] (2,1) -- (2,0) -- (2,-1);
        \end{tikzpicture}
        \subcaption*{$T$-type graph}
    } \qquad \qquad
    \parbox[b]{.25\textwidth}{\centering
        \begin{tikzpicture}[thick, scale=1.2, every node/.style={scale=0.8}]]
            \shade [shading=ball, ball color=black]  (1,1) circle (.07);
            \shade [shading=ball, ball color=black]  (2.5,1) circle (.07);
            \shade [shading=ball, ball color=black]  (1,2) circle (.07);
            \shade [shading=ball, ball color=black]  (1,3) circle (.07);
            \shade [shading=ball, ball color=black]  (1,4) circle (.07);
            \shade [shading=ball, ball color=black]  (2.5,2) circle (.07);
            \shade [shading=ball, ball color=black]  (2.5,3) circle (.07);
            \shade [shading=ball, ball color=black]  (2.5,4) circle (.07);
            \draw[thick] (1,1) -- (1,2) -- (1,3) -- (1,4);
            \draw[thick] (2.5,1) -- (2.5,2) -- (2.5,3) -- (2.5,4);
            \draw[thick] (1,3) -- (2.5,2);
        \end{tikzpicture}
        \subcaption*{$H$-type graph}
    }
    \caption{\ }
    \label{Fig:2}
\end{figure}

Recall that $\iv(G)$ is the number of internal vertices of $G$. It is clear that if $G$ is a $T$-type tree, then $\iv(G)=n-3$. And if $G$ is an $H$-type tree, then $\iv(G)=n-4$.

\begin{Theorem}
    \label{conj:2}
    Let $G$ be a tree on $[n]$ such that $J_G$ is an almost complete intersection ideal. Then, we have $\reg(S/J_G^t)=2t+\iv(G)-1$ for $t\ge 2$.
\end{Theorem}

\begin{proof}
    From \Cref{lem6}, we know that $G$ can be obtained by adding an edge between two disjoint paths, say $P_{n_1}$ and $P_{n_2}$, where $n_1+n_2=n$. Let $e$ be the edge between $P_{n_1}$ and $P_{n_2}$, then $G\setminus e=P_{n_1}\sqcup P_{n_2}$.
    Furthermore, from \cite[Theorems 4.1 and 4.2]{JNR}, we obtain that $\reg(S/J_G)=\iv(G)+1$.

    Now, we consider the following two cases by noticing that $\Ht(J_G)=n-2$.
    \begin{enumerate}[i]
        \item Suppose that $G$ is a $T$-type tree.
            Then $\reg(S/J_G)=n-2$. We claim that $\reg(S/(J_{G\setminus e}:f_e))=n-3$. With this,
            it follows from \Cref{thm:HT} that $\reg(S/J_G^t)=(n-3)+2t-1=2t+\iv(G)-1$ for $t\ge 2$.

            As for the regularity of $S/(J_{G\setminus e}:f_e)$, since $J_{G\setminus e}:f_e=J_{(G\setminus e)_e}$ from \cite[Theorem 3.7]{MR3169597}, it is enough to compute $\reg(S/J_{(G\setminus e)_e})$. Obviously, $(G\setminus e)_e$ is the disjoint union of a path (which we may assume to be $P_{n_2}$) and a $C_3$-type graph (which we may denote by $G'$, with the same vertex set as $P_{n_1}$). And this special $C_3$-type graph $G'$ is obtained by attaching paths to at most two of the three vertices of $C_3$. It follows from from \cite[Proposition 3.11]{MR3912960} and \Cref{thm:C3} that
            \[
                \reg(S/J_{(G\setminus e)_e})=\reg(S/P_{n_2})+\reg(S/J_{G'})=(n_2-1)+(2+n_1-4)= n-3,
            \]
            confirming the claim.
        \item If $G$ is an $H$-type tree, then $\reg(S/J_G)=n-3$. Thus, we can apply the first case of \Cref{cor:HT_general} with $\delta=2$ to obtain that $\reg(S/J_G^t)=(n-1)+2t-4=2t+\iv(G)-1$ for $t\ge 2$. \qedhere
    \end{enumerate}	
\end{proof}

\section{Graphs with almost complete intersection parity binomial edge ideals}

Within this section, we will study the regularity of powers of the parity binomial edge ideal $\calI_G$   of a graph $G$ where $\calI_G$ is an almost complete intersection. Without loss of generality, suppose that the vertex set of $G$ is $[n]$. Let $S=S_G=\kappa[x_i,y_i\mid 1\le i\le n]$ be a standard graded polynomial ring in $2n$ variables over a field $\kappa$. And for simplicity, if $e=\{i,j\}$ is an edge of a graph $G$, we write $\bar{g}_e=\bar{g}_{i,j}\coloneqq x_ix_j-y_iy_j$. Consequently, the parity binomial edge ideal $\calI_G=(\bar{g}_e\mid e\in E(G))$ in $S$. Furthermore, we will assume that $\kappa$ is an infinite field with $\Char(\kappa)\ne 2$.

Graphs whose parity binomial edge ideals are (almost) complete intersections, are classified in \cite[Section 3]{MR4186617}.

\begin{Lemma}
    [{\cite[Corollary 3.6]{MR4186617}}]
    \label{complete}
    Let $G$ be a simple graph. Then $\calI_G$ is a complete intersection if and only if all the bipartite connected components of $G$ are paths and non-bipartite connected components are odd cycles.
\end{Lemma}

\begin{Lemma}
    [{\cite[Theorems 3.7-3.11]{MR4186617}}]
    \label{complex}
    One can classify connected graphs whose parity binomial edge ideals are almost complete intersections. They are a subclass of trees, a subclass of unicyclic graphs, or a subclass of bicyclic graphs. Below is the complete list.
    \begin{enumerate}[a]
        \item \label{complex_a} Let $G$ be a connected bipartite graph. Then $\calI_G$ is an almost complete intersection ideal if and only if $G$ is either obtained by adding an edge between two disjoint paths or by adding an edge between two vertices of a path such that the girth of $G$ is even.
        \item \label{complex_b} Let $G$ be a connected odd unicyclic graph. Then $\calI_G$ is an almost complete intersection ideal if and only if $G$ is one of the following types:
            \begin{enumerate}[i]
                \item $G$ is obtained by adding an edge $e$ between an odd cycle and a path, or
                \item $G$ is obtained by adding an edge $e$ between two vertices of a path such that girth of $G$ is odd and at least one of the vertex is an internal vertex of the path, or
                \item $G$ is obtained by attaching a path of length $\ge 1$ to each vertex of a triangle.
            \end{enumerate}
        \item \label{complex_c} Let $G$ be a connected non-bipartite bicyclic cactus graph. Then $\calI_G$ is almost complete intersection if and only if $G$ is obtained by adding an edge $e$ between two disjoint odd cycles.
        \item \label{complex_d} Let $G$ be a connected graph which is obtained by adding a chord $e=\{u ,v\}$ in an odd unicyclic graph $H$ such that $\deg_{H} (u)=\deg_{H}(v)=2$. Then $\calI_G$ is almost complete intersection if and only if $H$ is an odd cycle.
        \item \label{complex_e} Let $G$ be a non-bipartite graph which is obtained by adding a chord $e=\{u ,v\}$ in an even unicyclic graph $H$ such that $\deg_{H} (u)=\deg_{H}(v)=2$. Then $\calI_G$ is almost complete intersection if and only if $H$ is one of the following:
            \begin{enumerate}[i]
                \item $H$ is an even cycle, or
                \item $H$ is obtained by attaching a path to a vertex $i$ of an even cycle such that $\{u,i\},\{v,i\}$ are edges of the even cycle.
            \end{enumerate}
    \end{enumerate}
\end{Lemma}

\begin{Lemma}
    [{\cite[Corollary 3.12]{MR4186617}}]
    Let $G=G_1\sqcup \cdots\sqcup G_k$ be a disjoint union of $k$ graphs. Then $\calI_G$ is almost complete intersection if and only if for some $i$, $\calI_{G_{i}}$ is almost complete intersection and for $j\neq i$, $\calI_{G_{j}}$ are complete intersections.
\end{Lemma}

Thus, all graphs whose parity binomial edge ideals are almost complete intersections are completely determined. Notice that if $\Char(\kappa)=0$, then all regularity computations of powers of ideals can be reduced to the case when $G$ is connected, by \cite[Lemma 4.1, Proposition 5.1]{MR3912960}. Therefore, in the following, we will only consider the case when $G$ is connected. 

Before the involved discussion, let us start by recalling the following fact from \cite{BMS}, connecting binomial edge ideals with parity binomial ideals.

\begin{Lemma}
    [{\cite[Corollary 6.2]{BMS}}]
    \label{Phi}
    Let $G$ be a bipartite graph with bipartition $[n]=V_1\sqcup V_2$ and $S=\kappa[x_1,\dots,x_n,y_1,\dots,y_n]$. Let $\Phi: S\mapsto S$ be the ring homomorphism determined by
    \[
        \Phi(x_i)=\begin{cases}
            x_i & \text{if $i\in V_1$},\\
            y_i &\text{if $i\in V_2$},
        \end{cases}
        \qquad
        \text{ and } \qquad
        \Phi(y_i)=\begin{cases}
            y_i & \text{if $i\in V_1$},\\
            x_i &\text{if $i\in V_2$}.
        \end{cases}
    \]
    Then $\Phi$ is an isomorphism and $\Phi(J_G)=\calI_G$.
\end{Lemma}

We will fix this isomorphism $\Phi$ throughout this section. And to compute the regularity of the powers of parity binomial edge ideals of some non-bipartite graphs, we need the following lemma from \cite{MR4186617}.

\begin{Lemma}
    [{\cite[Lemma 3.3]{MR4186617}}]
    \label{nonbipartite}
    Let $G$ be a non-bipartite graph on $[n]$. Assume that there exists some $e=\{u, v\}\in E(G)$ such that $G\setminus e$ is a bipartite graph. Then,
    \[
        \calI_{G\setminus e}: \bar{g}_e
        = \Phi(J_{(G\setminus e)_e}).
    \]
\end{Lemma}

\subsection{Bipartite case}
The easiest case is when $G$ is bipartite, being in item \ref{complex_a} of \Cref{complex}.

\begin{Theorem}
    Let $G$ be a bipartite graph on $[n]$ such that $\calI_G$ is an almost complete intersection ideal. Then, for $t\ge 2$, we have
    \[
        \reg(S/\calI_G^t)=
        \begin{cases}
            2t+\iv(G)-1 & \text{if $G$ is a tree},\\
            2t+n-4 &\text{if $G$ is a unicyclic graph}.
        \end{cases}
    \]
\end{Theorem}

\begin{proof}
    As described in \Cref{Phi}, we have $\reg(S/\calI_G^t)=\reg(S/J_G^t)$ for all $t\ge 1$. The assertion immediately follows from \Cref{thm:F3-reg}, \Cref{conj:2} and \Cref{complex} \ref{complex_a}.
\end{proof}

\subsection{Adding an edge between an odd cycle and a vertex of a path}
In this subsection, we deal with the first subcase of item \ref{complex_b} of \Cref{complex}.

When the designated vertex is an end vertex of the path, we get a balloon graph. This is also the case when we add an edge between two vertices of a path such that one of these two vertices is an end vertex of the corresponding path.

\begin{Theorem}
    Let $G$ be a balloon graph on $[n]$ having odd girth. Then $\reg(S/\calI_G^t)= 2t+n-3$ for $t\ge 1$.
\end{Theorem}

\begin{proof}
    Let $u$ be the vertex of degree $3$ in $G$ and $v$ be a neighbor of $u$ on the cycle. Set $e = \{u, v\}$ and then $G\setminus e$ is a path of length $n-1$. Thus, \cite[Theorem 4.5]{JKS} proved that $\reg(S/\calI_G)=n-1$. Notice that the natural generators of $\calI_{G\setminus e}$ form a regular sequence, while after appending $\bar{g}_e$, one gets a $d$-sequence of length $n$ by the proof of \cite[Theorem 3.8]{MR4186617}. Furthermore, $(G\setminus e)_e$ is obtained by attaching two paths to two vertices of a triangle respectively. Thus,
    \[
        \reg\left(\frac{S}{\calI_{G\setminus e}:\bar{g}_e}\right)=\reg\left(\frac{S}{J_{(G\setminus e)_e}}\right)=n-2
    \]
    by \Cref{thm:C3} and \Cref{nonbipartite}. Now, we may apply \Cref{thm:HT} and obtain that $\reg(S/\calI_G^t)= 2t+n-3$ for $t\ge 1$.
\end{proof}


Next, we consider the case when the designated vertex is an internal vertex of the path.

\begin{Theorem}
    Let $G$ be a connected graph on $[n]$ obtained by adding an edge between an odd cycle and an internal vertex of a path. Then $\reg(S/\calI_G^t)= 2t+n-4$ for $t\ge 1$.
\end{Theorem}

\begin{proof}
    By the proof of \cite[Theorem 3.8]{MR4186617}, $\calI_G$ is an almost complete intersection of height $n-1$.
    As \cite[Theorem 4.4]{JKS} already showed that $\reg(R/\calI_G)=n-2$, it follows that $\reg(S/\calI_G^t)=n+2t-4$ for $t\ge 1$ from the first case of \Cref{cor:HT_general}.
\end{proof}

\subsection{Adding an edge between two internal vertices of a path} 

In this subsection, we deal with the second subcase of item \ref{complex_b} of \Cref{complex}. The related graphs are also the $G_2$-type graphs that we have studied in the previous section.


\begin{Theorem}
    Let $G$ be a graph on $[n]$ obtained by adding an edge $e$ between two internal vertices of a path $P_n$ such that the girth of $G$ is odd. Then $\reg(S/\calI_G^t)=2t+n-4$ for $t\ge 1$.
\end{Theorem}

\begin{proof}
    We claim that $\reg(S/\calI_G)=n-2$. With this, since $\calI_G$ is an almost complete intersection of height $n-1$ by the proof of \cite[Theorem 3.8]{MR4186617}, it follows from the first case of \Cref{cor:HT_general} that $\reg(S/\calI_G^t)=2t+n-4$ for $t\ge 1$.

    As for the claim, we consider the following two cases.
    \begin{enumerate}[i]
        \item If the girth of $G$ is at least $5$, then we may choose an edge $e'$ on this induced odd cycle such that $e'\cap e=\emptyset$. It is clear that $G\setminus e'$ is an $H$-type tree that we have studied in the previous section. In particular, it is bipartite. Thus
            \[
                \reg(S/(\calI_{G\setminus e'}:\bar{g}_{e'}))=\reg(S/J_{(G\setminus e')_{e'}})=\reg(S/J_{G\setminus e'})=\reg(S/\calI_{G\setminus e'})=n-3,
            \]
            by  \Cref{conj:2}, \Cref{Phi} and \Cref{nonbipartite}. Applying item \ref{lem5-3} of \Cref{lem5} to the following exact sequence
            \[
                0\to \frac{S}{\calI_{G\setminus e'}:\bar{g}_{e'}}(-2) \xrightarrow{\cdot \bar{g}_{e'}} \frac{S}{\calI_{G\setminus e'}} \to \frac{S}{\calI_G} \to 0,
            \]
            we obtain $\reg(S/\calI_G)=n-2$, confirming the claim in this case.
        \item If the girth of $G$ is $3$, then
            we may assume that the triangle has the vertex set $\{1,2,4\}$ and $e=\{1,2\}$. Then,
            we will choose the edge $e''=\{2,4\}$ of the triangle. It is clear that $G\setminus e''$ is a $T$-type tree that we have studied in the previous section. In particular, it is bipartite. Thus
            \[
                \reg\left(\frac{S}{\calI_{G\setminus e''}}\right)=\reg\left(\frac{S}{J_{G\setminus e''}}\right)=n-2
            \]
            by \Cref{conj:2} and \Cref{Phi}, and
            \begin{equation*}
                \reg(S/(\calI_{G\setminus e''}:\bar{g}_{e''}))=\reg(S/J_{(G\setminus e'')_{e''}})
            \end{equation*}
            by \Cref{nonbipartite}. In the following, we will prove that $\reg(S/J_{(G\setminus e'')_{e''}})=n-3$.

            To show this, let $H$ be the graph on the set $[5]$ with edges
            \[
                \{1,2\},\{2,3\},\{1,3\},\{1,4\},\{1,5\}.
            \]
            Notice that $\widetilde{G}\coloneqq (G\setminus e'')_{e''}$ is (isomorphic to) the graph obtained from $H$ by adding a path to the vertices $3$ and $5$ respectively; see also \Cref{Fig:5}.
            \begin{figure}[htbp]%
                \centering
                \begin{tikzpicture}[thick, scale=1.2, every node/.style={scale=0.8}]]
                    \shade [shading=ball, ball color=black]  (0.5,3.8) circle (.07);
                    \shade [shading=ball, ball color=black]  (1.5,3.8) circle (.07);
                    \shade [shading=ball, ball color=black]  (2.5,3.8) circle (.07);
                    \shade [shading=ball, ball color=black]  (3.5,3.8) circle (.07) node [above] {\scriptsize$3$};
                    \shade [shading=ball, ball color=black]  (3,3) circle (.07) node [below] {\scriptsize$2$};
                    \shade [shading=ball, ball color=black]  (4,3) circle (.07) node [below] {\scriptsize$1$};
                    \shade [shading=ball, ball color=black]  (5,3) circle (.07) node [below] {\scriptsize$4$};
                    \shade [shading=ball, ball color=black]  (4.6,3.8) circle (.07) node [above] {\scriptsize$5$};
                    \shade [shading=ball, ball color=black]  (5.6,3.8) circle (.07);
                    \shade [shading=ball, ball color=black]  (6.6,3.8) circle (.07);
                    \shade [shading=ball, ball color=black]  (7.6,3.8) circle (.07);
                    \draw [dashed] (0.5,3.8) -- (1.5,3.8) -- (2.5,3.8) -- (3.5,3.8);
                    \draw [dashed] (4.6,3.8) -- (5.6,3.8) -- (6.6,3.8) -- (7.6,3.8);
                    \draw [thick] (4.6,3.8) -- (4,3) -- (3.5,3.8) -- (3,3) -- (4,3) -- (5,3);
                \end{tikzpicture}
                \caption{The graph $\widetilde{G}=(G\setminus e'')_{e''}$}
                \label{Fig:5}
            \end{figure}
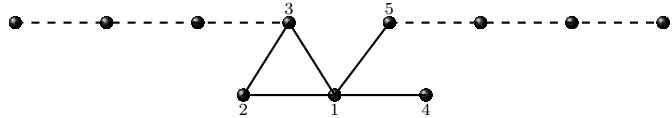
            By \Cref{lem:reduction_by_path}, it suffices to assume that these two paths do not appear. In other words, we may assume that $\widetilde{G}$ is precisely $H$.
            The last piece is to show that $\reg(S/J_{\widetilde{G}})=2$ in this case. But it has been confirmed by \cite[Theorem 3.2 (b)]{MR3859963}. In short, we have proved that $\reg(S/(\calI_{G\setminus e''}:\bar{g}_{e''}))=n-3$.

            Finally, we consider the following short exact sequence
            \[
                0\to \frac{S}{\calI_{G\setminus e''}:\bar{g}_{e''}}(-2) \xrightarrow{\cdot \bar{g}_{e''}} \frac{S}{\calI_{G\setminus e''}} \to \frac{S}{\calI_G} \to 0.
            \]
            It follows from item \ref{lem5-3} of \Cref{lem5} that $\reg(S/\calI_G)=n-2$, confirming the claim stated at the beginning. And this completes the proof. \qedhere
    \end{enumerate}
\end{proof}	

\subsection{Adding a path to each vertex of a triangle}
In this subsection, we deal with the third subcase of item \ref{complex_b} of \Cref{complex}. The related graphs are also the $C_3$-type graphs which have been studied in the previous section.

\begin{Theorem}
    Let $G$ be a graph on $[n]$ obtained by attaching a path of length $\ge 1$ to each vertex of a triangle. Then $\reg(S/\calI_G^t)=2t+n-4$ for $t\ge 2$.
\end{Theorem}

\begin{proof}
    Suppose that the triangle is $C_3$ with vertex set $\{1,2,3\}$. Let $e=\{2,3\}$ be an edge of this triangle. Then, $G\setminus e$ is a $T$-type tree with $\reg(S/\calI_{G\setminus e})=n-2$ by \Cref{Phi} and \Cref{conj:2}. Since $G\setminus e$ is bipartite, one has
    \[
        \reg(S/(\calI_{G\setminus e}:\bar{g}_e))=\reg(S/J_{(G\setminus e)_e}),
    \]
    which we claim to be $n-4$. To show this, let $H$ be the graph on the set $[6]$ with edges
    \[
        \{1,2\},\{2,5\},\{1,5\},\{1,3\},\{3,6\},\{1,6\},\{1,4\}.
    \]
    Then $\widetilde{G}\coloneqq (G\setminus e)_e$ is (isomorphic to) the graph obtained from $H$ by adding a path to the vertices $4$, $5$ and $6$ respectively; see also \Cref{Fig:6}.
    \begin{figure}[htbp]%
        \centering
        \begin{tikzpicture}[thick, scale=1.2, every node/.style={scale=0.8}]]
            \shade [shading=ball, ball color=black]  (3,2.2) circle (.07);
            \shade [shading=ball, ball color=black]  (2,2.2) circle (.07);
            \shade [shading=ball, ball color=black]  (1,2.2) circle (.07);
            \shade [shading=ball, ball color=black]  (0.5,3.8) circle (.07);
            \shade [shading=ball, ball color=black]  (1.5,3.8) circle (.07);
            \shade [shading=ball, ball color=black]  (2.5,3.8) circle (.07);
            \shade [shading=ball, ball color=black]  (3.5,3.8) circle (.07) node [above] {\scriptsize$5$};
            \shade [shading=ball, ball color=black]  (3,3) circle (.07) node [below] {\scriptsize$2$};
            \shade [shading=ball, ball color=black]  (4,3) circle (.07) node [above] {\scriptsize$1$};
            \shade [shading=ball, ball color=black]  (4,2.2) circle (.07) node [below] {\scriptsize$4$};
            \shade [shading=ball, ball color=black]  (5,3) circle (.07) node [below] {\scriptsize$3$};
            \shade [shading=ball, ball color=black]  (4.6,3.8) circle (.07) node [above] {\scriptsize$6$};
            \shade [shading=ball, ball color=black]  (5.6,3.8) circle (.07);
            \shade [shading=ball, ball color=black]  (6.6,3.8) circle (.07);
            \shade [shading=ball, ball color=black]  (7.6,3.8) circle (.07);
            \draw [dashed] (0.5,3.8) -- (1.5,3.8) -- (2.5,3.8) -- (3.5,3.8);
            \draw [dashed] (4, 2.2) -- (1, 2.2);
            \draw [dashed] (4.6,3.8) -- (5.6,3.8) -- (6.6,3.8) -- (7.6,3.8);
            \draw [thick] (4,3) -- (3.5,3.8) -- (3,3) -- (4,3) -- (5,3) -- (4.6, 3.8) -- (4,3) -- (4,2.2);
        \end{tikzpicture}
        \caption{The graph $\widetilde{G}=(G\setminus e)_{e}$}
        \label{Fig:6}
    \end{figure}
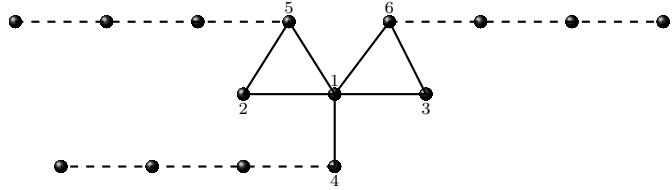
    By \Cref{lem:reduction_by_path}, it suffices to assume that these three paths do not exist and hence $\widetilde{G}=H$. And we are reduced to prove that $\reg(S/J_H)=2$, which is clear from \cite[Theorem 3.2 (b)]{MR3859963}. In short, we have proved that $\reg(S/(\calI_{G\setminus e}:\bar{g}_e))=n-4$.

    %
    Finally, we consider the following short exact sequence
    \[
        0\to \frac{S}{\calI_{G\setminus e}:\bar{g}_e}(-2)\xrightarrow{\cdot \bar{g}_{e}} \frac{S}{\calI_{G\setminus e}} \to \frac{S}{\calI_G} \to 0.
    \]
    It follows from item \ref{lem5-3} of \Cref{lem5} that $\reg(S/\calI_G)\le n-2$. With this, we can apply the first case of \Cref{cor:thm_3_general} and  item \ref{complex_b} of \Cref{complex} to obtain that $\reg(S/\calI_G^t)=2t+n-4$ for $t\ge 2$.
\end{proof}

\subsection{Adding a chord in an odd cycle}
In this subsection, we will consider the case when $G$ is a graph obtained by adding a chord in an odd cycle.
It is the graph given in item \ref{complex_d} of \Cref{complex}.

Before starting the discussion, we need a new notion from \cite{JKS}.

\begin{Definition}
    Let $G_1$ and $G_2$ be two subgraphs of a graph $G$. If $G_1\cap G_2$ is the complete graph $K_m$,
    then $G$ is called the \emph{clique sum} of $G_1$ and $G_2$ along $K_m$, denoted by $G_1\cup_{K_m} G_2$.
    If $m=2$, then it is the clique sum of $G_1$ and $G_2$ along an edge $e$ and will be denoted by $G_1\cup_{e} G_2$.
\end{Definition}

\begin{Theorem}
    Let $G$ be a graph obtained by adding a chord $e$ in an odd cycle $C_{n}$. Then $\reg(S/\calI_G^t)=2t+n-3$ for $t\ge 2$.
\end{Theorem}

\begin{proof}
    Let $e=\{u,v\}$ be a chord in the odd cycle $C_{n}$. Then $\calI_{G\setminus e}$ is a complete intersection by \Cref{complete}. Moreover, $\calI_{G\setminus e}:\bar{g}_e^2=\calI_{G\setminus e}:\bar{g}_e$. Hence, $\calI_G$ is  an almost complete intersection ideal generated by a quadratic $d$-sequence of length $n+1$.	 And we claim that $\reg(S/\calI_G)=n-2$. With this, we can then obtain $\reg(S/\calI_G^t)=2t+n-3$ for $t\ge 2$ by \Cref{thm-3}.

    Notice that $G$ is the clique sum of an odd cycle $C_{\text{odd}}$ and an even cycle $C_{\text{even}}$ along the chord $e$.
    Thus, to prove the claim, we need to distinguish into the following two cases.
    \begin{enumerate}[a]
        \item If the girth of this induced odd cycle $C_{\text{odd}}$ is at least $5$, then we take an edge $e'$ from this odd cycle such that $e'\cap e=\emptyset$. Thus $(G\setminus e')_{e'}=G\setminus e'$ is a $G_2$-type graph that we have discussed before. Meanwhile, it is bipartite. Therefore, $\reg(S/(\calI_{G\setminus e'}:\bar{g}_{e'}))=\reg(S/J_{(G\setminus e')_{e'}})=\reg(S/J_{G\setminus e'})=n-3$ by \Cref{nonbipartite} and the proof of \Cref{thm:F3-reg}. Simultaneously, $\reg(S/\calI_{G\setminus e'})=\reg(S/J_{G\setminus e'})=n-3$ since $G\setminus e'$ is bipartite. Thus, if we look at the short exact sequence
            \begin{equation}
                0\to \frac{S}{\calI_{G\setminus e'}:\bar{g}_{e'}}(-2)\xrightarrow{\cdot \bar{g}_{e'}} \frac{S}{\calI_{G\setminus e'}} \to \frac{S}{\calI_G} \to 0,
                \label{eqn:h}
            \end{equation}
            then we can obtain the claimed $\reg(S/\calI_G)=n-2$ by item \ref{lem5-3} of \Cref{lem5}.

        \item If the girth of the induced odd cycle $C_{\text{odd}}$ is $3$. We may assume that the edges of $G$ are
            $\{1,2\},\{2,3\},\dots,\{n-1,n\},\{1,n\},\{2,n\}$
            with $n\ge 5$ being an odd number. In this case, the chord $e=\{2,n\}$. For the edge $e'=\{1,n\}$ of the triangle, $G\setminus e'$ is a balloon graph on $[n]$ with an even girth. By \cite[Remark 3.15]{JKS}, we obtain $\reg({S}/{\calI_{G\setminus e'}})=\reg({S}/{J_{G\setminus e'}})=n-2$.

            Next, we will prove that $\reg (S/(\calI_{G\setminus e'}:\bar{g}_{e'}))=n-3$. It is obvious that $H\coloneqq (G\setminus e')_{e'}$ has edges $\{1,2\},\{2,3\},\ldots,\{n-1,n\},\{2,n-1\},\{2,n\}$.
            In the following, let $e''=\{1,2\}$.
            \begin{enumerate}[i]
                \item If $n=5$, then $(H\setminus e'')_{e''}$ is simply $K_4$. We have $\reg(S/J_{(H\setminus e''})_{e''})=1=n-4$ by \cite[Theorem 2.1]{MR2946102}.
                \item If $n\ge 7$, then $(H\setminus e'')_{e''}$ is the clique sum of the complete graph $K_4$ and the cycle $C_{n-3}$ along an edge. Thus, it follows from \cite[Proposition 3.11]{JKS} that $\reg(S/J_{(H\setminus e''})_{e''})=(n-3)-1=n-4$.
            \end{enumerate}
            Meanwhile, $H\setminus e''$ is the clique sum of $K_3$ and $C_{n-2}$ along an edge.
            Thus, $\reg(S/J_{H\setminus e''})=(n-2)-1=n-3$ by \cite[Proposition 3.11]{JKS}. Since $J_{H\setminus e''}:f_{e''}=J_{(H\setminus e'')_{e''}}$ by \cite[Theorem 3.7]{MR3169597}. Thus, it follows from the standard short exact sequence
            \[
                0\to \frac{S}{J_{H\setminus e''}:f_{e''}}(-2)\xrightarrow{\cdot f_{e''}} \frac{S}{J_{H\setminus e''}} \to \frac{S}{J_H} \to 0
            \]
            and item \ref{lem5-3} of \Cref{lem5} that $\reg(S/J_H)=n-3$.
            %
            Since $G\setminus e'$ is bipartite, it follows from \Cref{nonbipartite} that
            \[
                \reg\left(\frac{S}{\calI_{G\setminus e'}:\bar{g}_{e'}}\right)=\reg\left(\frac{S}{J_{(G\setminus e')_{e'}}}\right)=\reg\left(\frac{S}{J_{H}}\right)=n-3,
            \]
            as wished.

            Finally, we can obtain $\reg(S/\calI_G)=n-2$ by applying item \ref{lem5-3} of \Cref{lem5} to the following short exact sequence \eqref{eqn:h}, confirming the claim in this case. And this completes the proof. \qedhere
    \end{enumerate}
\end{proof}	

\subsection{Adding a chord in an even unicyclic graph }

In this final subsection, we will consider the case when $G$ is obtained by adding a chord in an even unicyclic graph. It is the graph given in item \ref{complex_e} of \Cref{complex}.

\begin{Theorem}
    Let $G$ be a non-bipartite graph on $[n]$ obtained by adding a chord in an even cycle $C_n$. Then $\reg(S/\calI_G^t)=2t+n-3$ for all $t\ge 1$.
\end{Theorem}

\begin{proof}
    We claim that $\reg(S/\calI_G)=n-1$. With this, since $\Ht(\calI_G)=n$, it then suffices to apply the first case of \Cref{cor:HT_general}.

    To prove the claim, let $e=\{u,v\}$ be the chord added to the even cycle $C_{n}$. As $G\setminus e=C_n$ is bipartite, $\reg(S/\calI_{G\setminus e})=\reg(S/J_{G\setminus e})=n-2$ by \Cref{Phi} and \cite[Thoerem 3.6]{JKS}. At the same time, $\calI_{G\setminus e}:\bar{g}_e\simeq J_{({G\setminus e})_e}$ by \Cref{nonbipartite}. Thus, it is sufficient to show that $\reg(S/J_{({G\setminus e})_e})=n-2$, since we can apply item \ref{lem5-3} of \Cref{lem5} to the following exact sequence
    \begin{equation}
        0\to \frac{S}{\calI_{G\setminus e}:\bar{g}_e}(-2)\xrightarrow{\cdot \bar{g}_{e}} \frac{S}{\calI_{G\setminus e}}\to \frac{S}{\calI_G}\to 0.
        \label{eqn:SES-e}
    \end{equation}
    We distinguish into the following two cases.
    \begin{enumerate}[a]
        \item If $n=4$, then $(G\setminus e)_e$ is isomorphic to $G$. It follows that $\reg(S/J_{(G\setminus e)_e})=2$ by \cite[Theorem 3.2]{MR3859963}.
        \item If $n\geq 6$, then $(G\setminus e)_e$ is a graph obtained by adding two distinct chords $e_1$ and $e_2$ in the cycle $C_{n}$.
            This graph can also be viewed as the clique sum of two complete graphs $K_3$ and $K_3$ along two distinct edges on $C_{n-2}$. Thus, it follows from
            \cite[Proposition 4.6]{MR4338055}
            that $\reg(S/J_{({G\setminus e})_e})=n-2$.
            \qedhere
    \end{enumerate}
\end{proof}

\begin{Theorem}
    Let $G$ be a non-bipartite graph on $[n]$ obtained by adding a chord $e=\{u ,v\}$ in an even cycle $H$ and a path to a vertex $i$ of the even cycle $H$ such that $\{u,i\},\{v,i\}$ are edges of the cycle $H$. Then $\reg(S/\calI_G^t)=2t+n-3$ for $t\ge 2$.
\end{Theorem}

\begin{proof}
    We claim that in this case $\reg(S/\calI_G)\le n-2$. With this, since $\Ht(\calI_G)=n$, we can apply the first case of \Cref{cor:thm_3_general} to achieve the desired result.

    To show the claim, notice that $G\setminus e$ is bipartite balloon graph. Thus,
    \[
        \reg(S/\calI_{G\setminus e})=\reg(S/J_{G\setminus e})=n-2
    \]
    by Remark \ref{Phi} and \cite[Remark 3.15]{JKS}. Meanwhile,
    $\reg(S/(\calI_{G\setminus e}:\bar{g}_e))=\reg(S/J_{H_1})$ for $H_1\coloneqq (G\setminus e)_e$.
    Thus, to confirm the claim by applying item \ref{lem5-3} of \Cref{lem5} to the short exact sequence \eqref{eqn:SES-e}, it suffices to show that
    \begin{equation}
        \reg(S/J_{H_1}) = n-4. \label{eqn:H_1}
    \end{equation}

    To show \eqref{eqn:H_1}, we first notice that by \Cref{lem:reduction_by_path}, we may suppose that the length of the path added to $H$ is $1$. Let $e'$ be the only edge of this path. It is clear that $J_{H_1\setminus e'}:f_{e'}=J_{(H_1\setminus e')_{e'}}$ by \cite[Theorem 3.7]{MR3169597}, where $(H_1\setminus e')_{e'}$ is a clique sum of $K_5$ and $C_{n-4}$ along an edge. Thus, $\reg(S/(J_{H_1\setminus e'}:f_{e'}))=n-5$ by \cite[Proposition 3.11]{JKS}.
    Meanwhile, $H_1\setminus e'$ is the clique sum of $C_{n-3}$ and two complete graphs $K_3$ and $K_3$ along two distinct edges on $C_{n-3}$. Thus, $\reg(S/J_{H_1\setminus e'})=n-4$ by
    \cite[Proposition 4.6]{MR4338055}.
    Therefore, by looking at the short exact sequence
    \[
        0\to \frac{S}{J_{H_1\setminus e'}:f_{e'}}(-2)\to \frac{S}{J_{H_1\setminus e'}} \to \frac{S}{J_{H_1}} \to 0,
    \]
    we see immediately that \eqref{eqn:H_1} holds.  And this completes the proof.
\end{proof}

\begin{Remark}
    We still have one case unsettled, which is when $G$ is obtained by adding an edge between two disjoint odd cycles, given in item \ref{complex_c} of \Cref{complex}. \texttt{Macaulay2} suggests that $\reg(S/\calI_{G})=|V(G)|-1$, which we don't know how to prove for the time being. But if we have this, then we can apply the first case of \Cref{cor:HT_general}.
\end{Remark}

\begin{acknowledgment*}
    The authors are grateful to the software system \texttt{Macaulay2} \cite{M2}, for serving as an excellent source of inspiration. The first author is partially supported by the ``Anhui Initiative in Quantum Information Technologies'' (No.~AHY150200). And the second author is supported by the National Natural Science Foundation of China (No.~11271275) and by foundation of the Priority Academic Program Development of Jiangsu Higher Education Institutions.
\end{acknowledgment*}

\bibliography{BEI}
\end{document}